\documentclass[12pt]{amsart}
\usepackage[utf8]{inputenc}

\usepackage[english]{babel}
\usepackage{amsmath, amssymb, amsthm, amscd,color,comment}
\usepackage{cancel}
\usepackage{cite}
\usepackage{alltt}
\usepackage[dvipsnames]{xcolor}
\usepackage{array}
\usepackage[small,bf,labelsep=period]{caption}
\usepackage{mathtools}

\makeatletter
\def\moverlay{\mathpalette\mov@rlay}
\def\mov@rlay#1#2{\leavevmode\vtop{%
		\baselineskip\z@skip \lineskiplimit-\maxdimen
		\ialign{\hfil$\m@th#1##$\hfil\cr#2\crcr}}}
\newcommand{\charfusion}[3][\mathord]{
	#1{\ifx#1\mathop\vphantom{#2}\fi
		\mathpalette\mov@rlay{#2\cr#3}
	}
	\ifx#1\mathop\expandafter\displaylimits\fi}
\makeatother

\newcommand{\cupdot}{\charfusion[\mathbin]{\cup}{\cdot}}

\usepackage{pgfplots}
\pgfplotsset{compat=1.15}
\usepackage{mathrsfs}
\usetikzlibrary{arrows}

\usepackage{enumerate}
\usepackage{cancel}
\newtheorem{theorem}{Theorem}[section]

\newtheorem{definition}[theorem]{Definition}

\newtheorem{lemma}[theorem]{Lemma}
\newtheorem{corollary}[theorem]{Corollary}
\newtheorem{proposition}[theorem]{Proposition}

\DeclarePairedDelimiter\ceil{\lceil}{\rceil}
\DeclarePairedDelimiter\floor{\lfloor}{\rfloor}

\def\la{\lambda}

\def\N{\mathbb N}
\def\R{\mathbb R}

\def\cX{\mathcal X}

\def\fq{\mathbb F_q}

\def\Ap{{\rm Ap}}

\def\char{\mbox{\rm Char}}

\definecolor{egreen}{RGB}{19, 126, 8}

\newcommand{\lu}[1]{{{\color{red}#1}}}

\everymath{\displaystyle}

\begin{document}
	\title[GM-bound on semigroups]{A closed formula for the Geil-Matsumoto bound on numerical semigroups via Apéry sets}
	\author{Adler Marques, Erik Mendoza, Luciane Quoos, Guilherme Tizziotti}
	\address{}
	\email{}
	\thanks{The third author was partially financed by Coordenação de Aperfeiçoamento de Pessoal de Nível Superior (CAPES): Finance Code 001, Conselho Nacional de Desenvolvimento Cient\'ifico e Tecnol\'ogico (CNPq): Project 307261/2023-9, and Fundação de Amparo a Pesquisa do Estado do Rio de Janeiro (FAPERJ): Project  E-26/204.037/2024. The fourth author was partially financed by CNPq Project 306270/2023-4, and Fundação de Amparo à Pesquisa do Estado de Minas Gerais (FAPEMIG) Project APQ-01443-23.}
	
	\address{Instituto de Matemática, Universidade Federal do Rio de Janeiro, Cidade Universitária,
		CEP 21941-909, Rio de Janeiro, Brazil}
	\email{adler@im.ufrj.br}
	
	\address{Instituto de Matemática, Universidade Federal do Rio de Janeiro, Cidade Universitária,
		CEP 21941-909, Rio de Janeiro, Brazil}
	\email{erik@im.ufrj.br}
	
	\address{Instituto de Matemática, Universidade Federal do Rio de Janeiro, Cidade Universitária,
		CEP 21941-909, Rio de Janeiro, Brazil}
	\email{luciane@im.ufrj.br}
	
	\address{Faculdade de Matemática, Universidade Federal de Uberlândia, Campus Santa Mônica, CEP 38400-902, Uberlândia, Brazil}
	\email{guilhermect@ufu.br}
	
	\begin{abstract} 
		The Geil-Matsumoto bound (GM bound) constrains  the number of rational points on a curve over a finite field in terms of the Weierstrass semigroup of any of the points on the curve. For general numerical semigroups, the GM bound lacks a simple closed-form expression, making its computation a challenging problem.  A closed formula has been obtained for the case when the semigroup is generated by two co-prime integers. In this work, for any numerical semigroup, we provide a closed formula for the GM bound in terms of the Apéry set of a nonzero element of the semigroup.
	In the case where the numerical semigroup is generated by consecutive integers $n, n+1, \dots, n+t$ with $\ceil*{\textstyle\frac{n-1}{2}}\leq t \leq n-1$, we obtain a simple closed formula for the bound. We apply these results to obtain upper bounds on the number of rational points for algebraic curves over finite fields. In some cases, our bounds improve some well-known upper bounds on the number of rational points. 
	\end{abstract}
	
	\maketitle
	
	\section{Introduction}
Let $\fq$ be the finite field with $q$ elements of characteristic $p$, and $ \fq (\mathcal{X})$ be the function field of an absolutely irreducible, non-singular algebraic curve $\cX$ defined over $\fq$ of genus $g$. We denote by $N_q(\cX)$  the number of rational points over $\fq$ (or rational places) on $\cX$ (on $ \fq (\mathcal{X})$). The study of curves with many rational points is a classical problem that has regained prominence due to its numerous applications in mathematics, see \cite{ADG2020, STV2006, BR2013, GM2009, NX2014, YL2025, YLZHH2024}, and coding theory, see \cite{BLV2014, CMQ2023, B2004} and references in there.

Several upper bounds for the number of rational points on algebraic curves are available in the literature.  A fundamental upper bound  is the Hasse-Weil bound, which states
$$N_q(\cX) \leq q+1+2g \sqrt{q},$$ see \cite[Theorem V.2.3]{S2009}. Serre remarked this  bound  could be improved in general to 
\begin{equation} \label{serreimproved}
N_q(\cX)\leq q+1+g\floor*{2 \sqrt{q}} \quad \quad \text{HWS-bound},
\end{equation}
see \cite[Theorem V.3.1]{S2009}.
 For curves of ``large genus", the bound established by  Ihara \cite{I1981} is better than  the Hasse-Weil in many cases, and is given by
\begin{equation}\label{Iharabound}
N_q(\cX) \leq q+1 + \frac{\sqrt{(8q+1)g^2+4(q^2-q)g}-g}{2} \quad \quad \text{Ihara-bound}.
\end{equation}

Given a rational point $Q$ on the curve $\cX $ of genus $g$, the \textit{Weierstrass semigroup} at $Q$ is defined by 
		$$
		H(Q)=\{s\in \N :  (z)_{\infty}=sQ\text{ for some }z\in \fq(\cX)\},	$$
		where $(z)_\infty$ denotes the pole divisor of the function $z$.
The complement set $ G(Q)=\N\setminus H(Q)$ is the set of gaps at $Q$. The Weierstrass Gap Theorem states that if $g >0$, then
		$
		\#G(Q)=g$, see \cite[Theorem 1.6.8]{S2009}. 
It is easy to show that for any $s_1, s_2 \in H(Q)$, their sum $s_1+s_2$ remains in $H(Q)$. This closure property under addition shows that $H(Q)$ forms a numerical semigroup. By definition, a  \textit{numerical semigroup} is an additive submonoid $S \subseteq \mathbb{N}$ for which the complement $\mathbb{N} \setminus S$ is finite. The cardinality of this complement, $g = \#( \mathbb{N} \setminus S)$, is called the \textit{genus} of the semigroup. 
		
In 1990,  Lewittes  obtained a new upper bound for the number of rational points on a curve over $\fq$ using a local information, the Weierstrass semigroup at one rational point $Q$ on the curve. Let $m_{H(Q)}$ be the least positive element in $H(Q)$, the {\it Lewittes upper bound} $L_q(H(Q))$ is given by
		\begin{equation}\label{Lbound original}
			N_q(\cX) \leq L_q(H(Q)) :=  qm_{H(Q)}+1,		
			\end{equation}
see \cite[Theorem 1]{L1990}. For $q=2,3,4$ and $g \geq 2$,  this bound improves the Hasse-Weil bound. 
The Hermitian curve is a famous example of a curve attaining the Lewittes' bound considering the Weierstras semigroup at the only point at infinity in the curve. In \cite[Theorem 5.3]{M2023}, Mendoza characterized for the specific family of Kummer extensions with one point $Q_\infty$ at infinity, the connections between being a Castle curve, attaining the Lewittes' bound for $H(Q_\infty)$, and the structure of the Weierstrass semigroup at $Q_\infty$.
		  
In 2009, Geil and Matsumoto generalized Lewittes' upper bound by incorporating arithmetic information about all the generators of the Weierstrass semigroup. Specifically, suppose $\{h_1<h_2< \cdots < h_m\}$ is a system of generators of the semigroup  $H(Q)$, the {\it Geil-Matsumoto bound } $GM_q(H(Q))$ is given by
		\begin{equation*}\label{GMbound original}
			\displaystyle N_q(\cX) \leq GM_q(H(Q)):= \# (H(Q)\setminus \bigcup_{i=1}^{m}(qh_i+H(Q)))+1,  
					\end{equation*}
where $qh_i+H(Q)=\{ qh_i+s : s \in H(Q)\}$ (see \cite[Theorem 1]{GM2009}). In the same paper 
	 they  also showed that the Lewittes' bound can be deduced from the Geil-Matsumoto bound, since 
		\begin{equation}\label{GMbetterL}
		 GM_q(H(Q))\leq \# (H(Q)\setminus (qh_1+H(Q)))+1 = q h_1 + 1=L_q(H(Q)).		 
		 \end{equation}
		 An interesting example of a curve attaining the GM bound is the the norm-trace curve over $\mathbb{F}_{q^r}$ with $2\leq r$,  introduced by Geil in \cite{G2003}. This  curve has $q^{2r-1}+1$ rational points over $\mathbb{F}_{q^r}$ and the Weierstrass semigroup $H$ at the only point at infinity on the curve is generated by two integers  $H= \langle q^{r-1}, q^{r-1}+ \dots +q+1 \rangle$. The equality of the GM bound follows from \cite[Theorem 3.2]{BV2014}.

The Lewittes' bound is significantly simpler than the Geil-Matsumoto bound. In fact, the latter lacks a closed-form expression for a general semigroup. This computational difficulty has led to an interesting theoretical problem:
		\begin{itemize}
			\item deriving a closed-form expression for the Geil-Matsumoto bound for arbitrary numerical semigroups.
		\end{itemize}

For a semigroup  generated by two (co-prime)  integers, Bras-Amor\'os and Vico-Oton \cite{BV2014} provided a closed formula for the Geil-Matsumoto bound. It is well known that given a semigroup generated by two elements, there exists a curve having a rational point whose Weierstrass semigroup is exactly this semigroup.  These semigroups have also been related with Fermat and Mersenne numbers, see \cite{ER2011}.
Semigroups generated by consecutive integers were studied by García-Sánchez and Rosales in \cite{RG1999} and \cite{R2000}, and also appear as Weierstrass semigroups in algebraic curves, see \cite{CMQ2016} and \cite{G2003}.

It is also interesting to mention that, despite their natural association with algebraic curves, not all numerical semigroups occur as the Weierstrass semigroup at a point on a projective nonsingular irreducible curve. In fact, let $S$ denote a numerical semigroup with gap set $\mathbb{N} \setminus S = \{\ell_1, \dots, \ell_g\}$ and let $S_m$ denote the set of all sums of $m$ elements of $S$. If $\# S_m \geq (2m-1)(g-1)$ for some $m \geq 1$, then $S$ cannot be the Weierstrass semigroup of a point on a projective nonsingular irreducible curve of genus $g$; see \cite{K1998}.  Explicit examples of semigroups that are not realizable as Weierstrass semigroups can be found in \cite{K1998, OS1997} and \cite{T1995}.
On the other hand, let $m_S$ denote the smallest positive element in a semigroup $S$ and let $\ell_g$ denote the largest gap of $S$. If $\ell_g < 2m_S$ and the weight of $S$ is at most $g-2$, then Eisenbud and Harris showed that $S$ does occur as the Weierstrass semigroup of a point on some curve; see \cite{EH1987ex}. 	Furthermore, in \cite{KY2013} Kaplan and Ye investigated the asymptotic proportion of numerical semigroups that are Weierstrass.

Motivated by these results, we investigated the Geil-Matsumoto bound and obtained two principal results, Theorems \ref{teo_GMbound} and \ref{GM formulas fechadas}. In Theorem \ref{teo_GMbound} we provide a closed expression for the Geil-Matsumoto bound in terms of the Apéry set of a nonzero element of the semigroup, and in Theorem \ref{GM formulas fechadas} we give a simple closed formula for the GM bound for semigroups generated by consecutive integers $n, n+1, \dots, n+t,$ where $\ceil*{\textstyle\frac{n-1}{2}}\leq t \leq n-1$. In Table \ref{table2} we provide examples of algebraic curves where the GM bound on the number of rational points  improves some well-known bounds.

The paper is organized as follows. In Section 2 we fix the notation that will be used throughout the paper and recall some basic facts about numerical semigroups and the Geil-Matsumoto bound. Section 3 is devoted to present a closed-form expression for the Geil-Matsumoto bound, while in Section 4 we study the Geil-Matsumoto bound for numerical semigroups generated by consecutive positive integers. Section 5 presents applications of the results established in previous sections to derive bounds for the number of rational points on certain Kummer extensions. Additionally, we provide comparisons with other known bounds.
		
	\section{Preliminaries and notation}	
	We start this section by establishing the notation that will be employed throughout this work.

At first we fix some arithmetical notation. For $a$ and $b$ integers, we denote by $(a, b)$ the greatest common divisor of $a$ and $b$. If $a \geq 1$ we let   ``$b\bmod{a}$" stand for the remainder of the division of $b$ by $a$. For $c\in \R$, we denote by $\floor*{c}$, $\ceil*{c}$, and $\{c\}$ the floor, ceiling and fractional part functions of $c$, respectively. The set of nonnegative integers is denoted by $\mathbb{N}.$ For $\alpha \in \mathbb{N}$ and a subset $A$ of $\mathbb{N}$, we use the standard notation for the translation of $A$  by $\alpha$, that is,  $A +\alpha :=\{ \beta + \alpha : \beta \in A \}.$

 A \textit{numerical semigroup} is an additive submonoid $S \subseteq \mathbb{N}$ such that its complement $G(S):= \mathbb{N} \setminus S$ is finite. The elements in  $G(S)$ are called {\it gaps}, and $g:=\#G(S)$ is the genus of the semigroup.  The smallest nonzero element of $S$, denoted by $m_S$, is the \textit{multiplicity} of the numerical semigroup $S$. It is well known that a numerical semigroup is finitely generated, that is, there exist $h_1,\ldots, h_t \in S^*=S \setminus \{0\}$ coprimes such that 
$$S = \langle h_1,\ldots, h_t \rangle : = \left\{ \sum_{j=1}^{t} h_j n_j : n_j \in \mathbb{N} \mbox{ for } j=1,\ldots,t \right\}.$$ 
The set $\{h_1, \ldots , h_t\}$ is called a \textit{system of generators} for $S$.  A distinguished family of generators for a semigroup arises from Apéry sets, a fundamental construction in semigroup theory. We now introduce this concept.
		\begin{definition}
			Let $S$ be a numerical semigroup and $n$ be a nonzero element of $S$. The Ap\'{e}ry set of $n$ in $S$ is defined by
			$$
			\Ap (S,n) := \{s\in  S : s-n \notin S\}.
			$$  
		\end{definition}				
García-Sánchez and Rosales  proved  a method for constructing the Apéry set.				
		\begin{lemma}\cite[Lemma 2.4, Lemma 2.6]{RG2009}\label{lema_Apery}
		Let $S$ be a numerical semigroup and let $n$ be a nonzero element of $S$. Then 
		\begin{enumerate}
			\item $\Ap(S,n)=\{0=w(0), w(1), \dots, w(n-1)\},$ where $w(i)$ is the least element of
			$S$ congruent to $i$ modulo $n$, for $i=1, \dots, n$. 
			\item  The set $\{n\}\cup (\Ap(S, n)\setminus \{0\})$ is a system of generators for $S$. 
		\end{enumerate}
		\end{lemma}
In \cite{RG1999} and \cite{R2000}, the same authors investigated generators, and minimal generators sets of numerical semigroups. Of particular interest is their result concerning semigroups generated by consecutive integers.
\begin{lemma} \cite[Corollary 4]{RG1999} \label{Apery consecutive}
	Let $S = \langle n, n+1, \ldots , n+t \rangle$ be a numerical semigroup with $1 \leq t < n$. Then $$\Ap(S, n) = \{ \lambda n + (\lambda - 1) t + r : 1 \leq r \leq t, \, \lambda \in \mathbb{N},\, \mbox{and } 0 \leq (\lambda - 1) t + r < n \}.$$
\end{lemma}

Surprisingly, the theory of numerical semigroups has a deep connection with the theory of algebraic curves over finite fields through Weierstrass semigroups. This relation yields a particularly interesting result, the fact that the number of rational points on an algebraic curve over a finite field can be bounded using properties of the Weierstrass semigroup at one point of the curve \cite{L1990}. We now introduce the relevant notation and recall several fundamental results connecting these theories.

An {\it algebraic curve $\cX$ over $\fq$} is a projective, nonsingular,  absolutely irreducible of genus $g$. 
Let $\mathbb{F}_q(\mathcal{X})$ be the function field of $\mathcal{X}$ and $\mathcal{P}_{\mathbb{F}_q(\mathcal{X})}$ denote the set of places of $\mathbb{F}_q(\mathcal{X})$. The number of rational places (degree one places) of $\mathbb{F}_q(\mathcal{X})$, or equivalently, the number of rational points on $\mathcal{X}$, is denoted by $N_q(\mathcal{X})$.
The study of curves with many points and  bounds (lower and upper) for $N_q(\mathcal{X})$ dates back to the work of Gauss and Fermat and has been extensively investigated,  see \cite{L1990, I1981, VD1983} for results from eighties, and for more recent works on the subject see \cite{ GM2009, GMQ2023, BR2013, HMP2025, XY2007}.

The Lewittes and the Geil-Matsumoto bounds can be studied in greater generality by extending the analysis to arbitrary numerical semigroups.
Let $q$ be a positive integer and $S$ a  numerical semigroup  with multiplicity $m_S$. We define the Lewittes bound of $S$ as
		\begin{equation}\label{Lbound}
			L_q(S) := qm_{S}+1  \quad \text{ Lewittes bound,}
		\end{equation}
		and for $qS^*+S:=\{q a + b \mbox{ : } a \in S^* \mbox{ and } b \in S\}$ we denote the {\it Geil-Matsumoto bound} of $S$ by 
		\begin{equation}\label{GMbound}
			GM_q(S):= \# (S\setminus (qS^*+S))+1 \quad \text{ GM bound. }
		\end{equation}
		
Given a numerical semigroup $S$, Equation~\eqref{GMbetterL} establishes that $GM_q(S) \leq L_q(S)$. Some authors have investigated conditions for equality in this bound.  In \cite[Proposition 9]{BR2013}, Beelen and Ruano proved that for any semigroup $S$, the equality $GM_q(S) = L_q(S)$ holds whenever $q\in S$. In particular, 
			\begin{equation}\label{eq}
			\text{there exists $a \in S$ such that $a$ divides $q$ $\quad \Rightarrow\quad  GM_q(S) = L_q(S)$. }
			\end{equation}
			In 2014, Bras-Amor\'os and Vico-Oton provided a necessary and sufficient condition for the Geil-Matsumoto and Lewittes bounds to coincide. More specifically, they established the following result.
			\begin{theorem}\cite[Theorem 4.1]{BV2014}\label{teo_GM=L}\label{Maria}
				Suppose that $S= \langle h_1, h_2, \dots, h_t \rangle$, where $h_1$ is the multiplicity of the semigroup $S$. Then
				$$\label{GM=L}
				GM_q(S)=L_q(S)\quad \Leftrightarrow \quad q(h_i-h_1)\in S \text{ for every }2\leq i \leq t.
				$$
		\end{theorem}

In contrast with the Lewittes' bound simpler formulation, for an arbitrary numerical semigroup $S$ the Geil-Matsumoto bound $GM_q(S)$ lacks a general closed-form expression and can be a challenging problem to obtain it. In \cite{BV2014}, Bras-Amorós and Vico-Oton established  an explicit closed-form expression for $GM_q(S)$ in the special case where $S$ is a numerical semigroup generated by two integers.
				\begin{theorem}\cite[Theorem 3.2]{BV2014}\label{teo_GM_twogenerators}
Let $a<b$ with $(a, b)=1$, then
\begin{equation*}
	GM_q(\langle a,b \rangle) =  1 + \sum_{k=0}^{a-1} \min \left\{ q, \ceil*{ \dfrac{q-k}{a} } b \right\}.
\end{equation*}
				\end{theorem} 
In the following sections, we present several results concerning the Geil--Matsumoto bound for numerical semigroups.

\section{Geil-Matsumoto bound for numerical semigroups via Apéry sets}

In this section, we first present a new simple condition on the semigroup $S$ that guarantees the equality $GM_q(S) = L_q(S)$. We then derive a closed-form expression for the Geil-Matsumoto bound using the Apéry set of a nonzero element of the semigroup.
 
Let $S$ be a numerical semigroup and $n \in S^*$. By Lemma \ref{lema_Apery}, $$\Ap(S, n)=\{0=w(0), w(1), \dots, w(n-1)\},$$ where $w(i)$ is the least element of $S$ congruent with $i$ modulo $n$. Furthermore, from the definition of $w(i)$ and writing $\gamma(i):=(w(i)-i)/n$ for $0\leq i \leq n-1$, we can rewrite the semigroup $S$  and as 
		\begin{equation}\label{eq_H}
		S=\left\{jn+i: 0\leq i \leq n-1,\, \gamma(i)\leq j\right\}.
		\end{equation}
		Moreover, by Lemma \ref{lema_Apery}, we have that
		\begin{equation}\label{generators}
		S=\langle n, w(1), w(2) \dots, w(n-1)\rangle.
		\end{equation}
		
Using Apéry sets, we can prove a similar result as in (\ref{eq}), that is, if the multiplicity $m_S$ divides $q-1$, then the Lewittes' bound and the Geil-Matsumoto bound coincide.
		\begin{proposition} \label{igualdade cotas}
		Let $S$ be a numerical semigroup and $m_S$ its multiplicity. If $m_S$ divides $q-1$, then $GM_q(S) = L_q(S)$. 
		\end{proposition}
		\begin{proof}
		Let $2\leq n \in S$, and $S=\langle n, w(1), w(2),  \dots, w(n-1) \rangle$ as in (\ref{generators}).
	Then, 
			\begin{equation*}\label{membershipS}
				s \in S \quad \Leftrightarrow \quad s \geq w(s \bmod n).
			\end{equation*}
			Thus, we have that for each $1 \leq i \leq m_S-1$,
			\begin{align*}
				q(w(i)-m_S) \in S &\Leftrightarrow q(w(i)-m_S) \geq w((q(w(i)-m_S)) \bmod m_S) \\
				&\Leftrightarrow q(w(i)-m_S) \geq w((q(\gamma(i)n + i - m_S)) \bmod m_S) \\
				&\Leftrightarrow q(w(i)-m_S) \geq w((qi) \bmod m_S).
			\end{align*}
			Therefore, from Theorem \ref{teo_GM=L}, we conclude that
			\begin{equation}\label{GM=L}
				GM_q(S) = L_q(S) \quad \Leftrightarrow \quad q(w(i)-m_S) \geq w((qi) \bmod m_S) \text{ for every } 1 \leq i \leq m_S-1.
			\end{equation}
			Since $m_S$ divides $q-1$, then for every $1 \leq i \leq m_S-1$,
			$$
			w((qi) \bmod m_S) = w(((q-1)i + i) \bmod m_S) = w(i).$$ 
			This yields
			$$
			q(w(i)-m_S) - w(i) = (q-1)w(i) - qm_S \geq (m_S+1)(q-1) - qm_S = q-1-m_S \geq 0.
			$$
			From (\ref{GM=L}) we conclude that $GM_q(S) = L_q(S)$.
		\end{proof}
		
		From Proposition \ref{igualdade cotas} one could wonder if in a semigroup $S$ we also have the following similar result: if  $m_S$ is a divisor of $q+1$, then $GM_q(S) = L_q(S)$. This is not true,  since for $q=11$ and  $S = \langle 6,7 \rangle$  we have $m_S=6$ is a divisor of $12$ but  $GM_q(S) \neq  L_q(S)$ from Theorem \ref{Maria} (notice that $11 \notin S$).

		In what follows, we present a closed formula for the Geil-Matsumoto bound $GM_q(S)$ for a numerical semigroup $S$ using Apéry sets. 
		First we need a technical lemma. 
		
		\begin{lemma}\label{lema_GM}
		If $S=\langle h_1, h_2, \dots, h_t\rangle$ is the semigroup generated by $h_1, h_2, \dots , h_t$, then 
		$$
		GM_q(S)=\# \left\{s\in S: \begin{array}{l}
			s-qh_1\notin S\\
			s-qh_2\notin S\\
			\hspace{1.1cm}\vdots\\
			s-qh_t\notin S 
		\end{array}
		\right\}+1.
		$$
		\end{lemma}	
		\begin{proof}We recall that $GM_q(S):= \# (S\setminus (qS^*+S))+1.$ Let $s \in S\setminus (qS^*+S)$. Suppose that $s - q h_j \in S$ for some $j \in \{1,\ldots , t\}$. We can write $s = qh_j + (s - q h_j)$ and so $s \in qS^*+S$, since $h_j \in S$, a contradiction. 
		Conversely, let $s \in S$ such that $s - q h_j \notin S$ for all $j \in \{1,\ldots , t\}$. If $s \in qS^*+S$, then $s=qs_1 + s_2$, for some $s_1 \in S^*$ and $s_2 \in S$. Since $S=\langle h_1, h_2, \dots, h_t\rangle$ and $h_i \in S^* $ for every $1\leq i \leq t$, there exists $j \in \{ 1,\ldots, t\}$ such that $s_1 = h_j + s_3$, with $s_3 \in S$. Then, we have $s = qh_j + qs_3 + s_2$ that implies $s - q h_j \in S$, a contradiction.
		\end{proof}

		\begin{theorem}\label{teo_GMbound}
		Let $S$ be a numerical semigroup, $n \in S^*$, and consider the Apéry set 
		$$\Ap(S, n)=\{0=w(0), w(1), \dots, w(n-1)\}.$$ 
		Then $S=\langle n, w(1), w(2) \dots, w(n-1)\rangle$, and for  $k_i:=(k-iq)\bmod{n}$ we have
		$$
		GM_q(S)=1+\frac{1}{n}\sum_{k=0}^{n-1}\min_{1\leq i \leq n-1}\left\{qn, qw(i)-w(k)+w(k_i)\right\}.
		$$
		\end{theorem}
		\begin{proof}
		For each $0\leq k\leq n-1$, define the set 
		$$S_k:=\{s\in S: s\equiv k\bmod{n}, \, s-qn\notin S, \text{ and }s-qw(i)\not\in S\text{ for every }1\leq i \leq n-1\}.$$	
		From Lemma \ref{lema_GM}, we have that $GM_q(S)=1+\textstyle\sum_{k=0}^{n-1} \# S_k$. Now, we compute the cardinality of $S_k$ for each $k$. Let $s\in S_k$. By the description of $S$ given in (\ref{eq_H}), we have that $s=jn+k$, where $\gamma(k)\leq j$. Furthermore,
		\begin{align*}
		s-qn\not\in S &\Leftrightarrow jn+k-qn\not\in S\\ 
		&\Leftrightarrow (j-q)n+k\not\in S\\
		&\Leftrightarrow j-q\leq \gamma(k)-1 &\text{by Equation (\ref{eq_H})}\\
		&\Leftrightarrow j\leq q-1+\gamma(k)
		\end{align*}
		and, for $1\leq i \leq n-1$, using that $w(i)=\gamma(i)n+i$, we get
		\begin{align*}
		s-qw(i)\not\in S &\Leftrightarrow jn+k-q\left(\gamma(i)n+i\right)\not\in S\\
		&\Leftrightarrow \left(j-q\gamma(i)\right)n+k-iq\not\in S\\
		&\Leftrightarrow \left(j-q\gamma(i)+\floor*{\frac{k-iq}{n}}\right)n+k_i\not\in S\\
		&\Leftrightarrow j-q\gamma(i)+\floor*{\frac{k-iq}{n}}\leq \gamma (k_i)-1\\
		&\Leftrightarrow j\leq q\gamma(i)+\gamma (k_i)-\floor*{\frac{k-iq}{n}}-1.
		\end{align*}
		Thus, $s=jn+k\in S_k$ if and only if 
		$$
		\gamma(k)\leq j \leq \min\left\{q-1+\gamma(k), \min_{1\leq i \leq n-1}\left\{q\gamma(i)+\gamma (k_i)-\floor*{\frac{k-iq}{n}}-1\right\}\right\}.
		$$ 
		Hence,
				\begin{align*}
		\# S_k
		&=\max\left\{0, \min\left\{q, \min_{1\leq i \leq n-1}\left\{q\gamma(i)-\gamma(k)+\gamma (k_i)-\floor*{\frac{k-iq}{n}}\right\}\right\} \right\}.
		\end{align*}
		Note that, for $1\leq i \leq n-1$,
		\begin{align*}
		q\gamma(i)-\gamma(k)+\gamma (k_i)-\floor*{\frac{k-iq}{n}}&=q\gamma(i)-\gamma(k)+\gamma (k_i)-\frac{k-iq-k_i}{n}\\
		&=\frac{q(n\gamma(i)+i)-(n\gamma(k)+k)+n\gamma(k_i)+k_i}{n}\\
		&=\frac{qw(i)-w(k)+w(k_i)}{n}.
		\end{align*}
		This implies that
		\begin{align*}
		\# S_k&=\max\left\{0, \min\left\{q, \min_{1\leq i \leq n-1}\left\{q\gamma(i)-\gamma(k)+\gamma (k_i)-\floor*{\frac{k-iq}{n}}\right\}\right\} \right\}\\
		&=\max\left\{0, \min\left\{q, \min_{1\leq i \leq n-1}\left\{\frac{qw(i)-w(k)+w(k_i)}{n}\right\}\right\} \right\}\\
		&=\frac{1}{n}\max\left\{0, \min\left\{qn, \min_{1\leq i \leq n-1}\left\{qw(i)-w(k)+w(k_i)\right\}\right\} \right\}.
		\end{align*}
		To finish the proof, we prove that $qw(i)-w(k)+w(k_i) \geq 0$ for every $1\leq i \leq n-1$. In fact, for $1\leq i \leq n-1$, we can write
		\begin{align*}
		qw(i)+w(k_i)&=q(n\gamma(i)+i)+n\gamma(k_i)+k_i\\
		&=q(n\gamma(i)+i)+n\gamma(k_i)+k-iq-\floor*{\frac{k-iq}{n}}n\\
		&=\left(q\gamma(i)+\gamma(k_i)-\floor*{\frac{k-iq}{n}}\right)n+k.
		\end{align*}
		 Since $w(i)$ and $w(k_i)$ are elements of the semigroup $S$, we have that $qw(i)+w(k_i)$ is an element of $S$ congruent with $k$ modulo $n$. Now, from the definition of $w(k)$, $$qw(i)+w(k_i)\geq w(k).$$ Thus, we conclude that $qw(i)-w(k)+w(k_i) \geq 0$ for every $0\leq i \leq n-1$. Consequently, $$\# S_k=\frac{1}{n}\min_{1\leq i \leq n-1}\left\{qn, qw(i)-w(k)+w(k_i)\right\},$$ and the proof is complete.
		\end{proof}
		\begin{corollary}\label{coro_GMbound}
		With notation as in  Theorem \ref{teo_GMbound}, if $\{s_1, s_2, \dots, s_r\} \subseteq \{1, \dots, n-1\}$ is such that 
		$$
		S=\langle n, w(s_1), w(s_2), \dots, w(s_r)\rangle, 
		$$
		then
		$$
		GM_q(S)=1+\frac{1}{n}\sum_{k=0}^{n-1}\min_{i\in\{s_1, s_2, \dots, s_r\}}\left\{qn, qw(i)-w(k)+w(k_i)\right\}.
		$$
		\end{corollary}
		\begin{proof}
		The proof follows straightforward from the proof of Theorem \ref{teo_GMbound}.
		\end{proof}

		Now, we show that we can recover the closed-form expression for the Geil-Matsumoto bound $GM_q(S)$ given in Theorem \ref{teo_GM_twogenerators} for $S=\langle a, b\rangle$ with $(a, b)=1$. Note that we do not assume the condition $a<b$. 
		
		\begin{corollary}\label{GMq(S)_twogenerators} Let $S=\langle a, b \rangle$ be the semigroup generated by two elements with $(a, b)=1$. Then 
		$$GM_q(S)=1+\sum_{k=0}^{a-1}\min\left\{q,\ceil*{\frac{q-k}{a}}b\right\}.$$
		\end{corollary}
		\begin{proof}It is well known that the Apéry set of $a$ in $S$ is given by
		$$
		\Ap(S, a)=\{0, b, 2b, \dots, (a-1)b\}.
		$$
		Therefore $w(i)=b((ib')\bmod{a})$ for $0\leq i \leq a-1$, where $b'\in \{1, 2, \dots, a-1\}$ is such that $bb'\equiv 1 \bmod{a}$. Moreover, observe that $w(b\bmod{a})=b$. Thus, from Corollary \ref{coro_GMbound},
		\begin{align*}
		GM_q(S)&=1+\frac{1}{a}\sum_{k=0}^{a-1}\min_{i\in\{b\bmod{a}\}}\left\{qa, qw(i)-w(k)+w(k_i)\right\}\\
		&=1+\frac{1}{a}\sum_{k=0}^{a-1}\min\left\{qa, qb-w(k)+w(k_{b\bmod{a}})\right\}.
		\end{align*}
		Now, note that
		\begin{align*}
		qb-w(k)+w(k_{b\bmod{a}})&=qb-b((kb')\bmod{a})+b((kb'-q)\bmod{a})\\
		&=qb-b\left(kb'-\floor*{\frac{kb'}{a}}a\right)+b\left(kb'-q-\floor*{\frac{kb'-q}{a}}a\right)\\
		&=\ceil*{\frac{q}{a}-\left\{\frac{kb'}{a}\right\}}ab.
		\end{align*}
		Since $(a, b')=1$, we conclude that
		\begin{align*}
		GM_q(S)&=1+\frac{1}{a}\sum_{k=0}^{a-1}\min\left\{qa,\ceil*{\frac{q}{a}-\left\{\frac{kb'}{a}\right\}}ab \right\}\\
		&=1+\sum_{k=0}^{a-1}\min\left\{q,\ceil*{\frac{q}{a}-\left\{\frac{kb'}{a}\right\}}b\right\}\\
		&=1+\sum_{k=0}^{a-1}\min\left\{q,\ceil*{\frac{q}{a}-\frac{k}{a}}b\right\}\\
		&=1+\sum_{k=0}^{a-1}\min\left\{q,\ceil*{\frac{q-k}{a}}b\right\}.
		\end{align*}
		\end{proof}

\section{GM bound for semigroups generated by intervals}

In this section, we study the Geil-Matsumoto bound for numerical semigroups generated by consecutive positive integers, that is, semigroups of the form $S=\langle n, n+1, \dots, n+t \rangle$ with $1 \leq t < n$. 
We begin by noting that if $S=\langle n, n+1, \dots, n+t \rangle$ with $1 \leq t < n$ then, by Theorem \ref{teo_GM=L}, the following equivalences hold:
\begin{equation}
GM_q(S)=L_q(S)\quad  \Leftrightarrow \quad  qi\in S \text{ for every }1\leq i \leq t \quad  \Leftrightarrow \quad q\in S. 
\end{equation}
Moreover, from \cite[Corollary 2]{RG1999}, we have that $q\in S$ if and only if $(q\bmod{n})\leq \floor*{\textstyle\frac{q}{n}}t$. In particular, if $q <n$, we always have $
GM_q(S)<L_q(S)$. Consequently, for this class of semigroups we have
$$GM_q(S)<L_q(S)\quad \Leftrightarrow \quad \floor*{\frac{q}{n}}t < (q\bmod{n}).$$
In the following theorem, we provide a  formula for the Geil-Matsumoto bound for  semigroups generated by consecutive integers.	
		\begin{theorem}\label{teo_GMbound consecutivos}
		Let $S=\langle n, n+1, \dots, n+t \rangle$ with $1 \leq t < n$. Then
		\begin{equation} \label{bound consecutivos}
		GM_q(S)=1+\sum_{k=0}^{n-1}\min_{1\leq i \leq t}\left\{q, C_{k,i} \right\},
	\end{equation}
		where $k_i=(k-iq)\bmod{n}$ and $C_{k,i}:=q-\ceil*{\frac{k}{t}}+\ceil*{\frac{k_i}{t}}+\ceil*{\frac{iq-k}{n}}$.
		\end{theorem}
		\begin{proof}
		Let $S=\langle n, n+1, \dots, n+t \rangle$ with $1 \leq t < n$. By Lemma \ref{Apery consecutive}, we obtain
		\begin{align*}
		\Ap(S, n) &= \{ \lambda n + (\lambda - 1)t+r:\,  1 \leq r \leq t, \,\, \lambda \in \mathbb{N},\, \mbox{and } 0 \leq (\lambda - 1) t + r < n \}\\
		&= \{ \lambda n + k:\,  1 \leq r \leq t, \,\, \lambda \in \mathbb{N},\, 0 \leq k < n, \text{ and } k=(\lambda - 1) t + r\}\\
		&= \left\{ \lambda n + k:\,  1 \leq r \leq t, \,\, \lambda \in \mathbb{N},\, 0 \leq k < n, \text{ and } \la=\ceil*{\textstyle\frac{k}{t}}\right\}\\
		&= \{\ceil*{\textstyle\frac{k}{t}}n + k:\, \, 0 \leq k < n \}.
		\end{align*} 
		 So, for every $0\leq k \leq n-1$ we have 
		$w(k) = \ceil*{\textstyle\frac{k}{t}}n + k .$ Since  $S=\langle n, w(1), w(2), \dots, w(t)\rangle$,
		Corollary \ref{coro_GMbound} together with the following equality $$k-iq - k_i = n \left\lfloor \frac{k - iq}{n} \right\rfloor=-n\left\lceil \frac{iq-k}{n} \right\rceil$$ yields
		\begin{align*}
			GM_q(S)&=1+\frac{1}{n}\sum_{k=0}^{n-1}\min_{1\leq i \leq t}\left\{qn, q \left\lceil \frac{i}{t} \right\rceil n + iq-\left\lceil \frac{k}{t} \right\rceil n - k + \left\lceil \frac{k_i}{t} \right\rceil n + k_i\right\}\\
			&=1+\frac{1}{n}\sum_{k=0}^{n-1}\min_{1\leq i \leq t}\left\{qn, qn + iq-\left\lceil \frac{k}{t} \right\rceil n - k + \left\lceil \frac{k_i}{t} \right\rceil n + k_i\right\}\\
			&=1+\frac{1}{n}\sum_{k=0}^{n-1}\min_{1\leq i \leq t}\left\{qn, qn -\left\lceil \frac{k}{t} \right\rceil n + \left\lceil \frac{k_i}{t} \right\rceil n + n \left\lceil \frac{iq-k}{n} \right\rceil \right\}\\
			&=1+\sum_{k=0}^{n-1}\min_{1\leq i \leq t}\left\{q, q -\left\lceil \frac{k}{t} \right\rceil + \left\lceil \frac{k_i}{t} \right\rceil+ \left\lceil \frac{iq-k}{n} \right\rceil \right\}.
		\end{align*}
		\end{proof}
		
		In the follow, we aim to establish conditions under which the expression for $GM_q(S)$, given in the previous result, can be simplified. To achieve this, we analyze the possible values of $$C_{k,i}:=q-\ceil*{\frac{k}{t}}+\ceil*{\frac{k_i}{t}}+\ceil*{\frac{iq-k}{n}}.$$ 
		We begin by noting that $\ceil*{\textstyle\frac{iq-k}{n}} \geq 0$, since $0\leq k \leq n-1$ and $1\leq i \leq t$. Therefore, in the case $k=0$, we have
		\begin{equation}\label{eq1}
		C_{0, i}=q+\ceil*{\frac{0_i}{t}}+\ceil*{\frac{iq}{n}}\geq q \quad \text{for every}\quad 1\leq i\leq t. 
		\end{equation}
	
	\begin{proposition} \label{valores Cik}
	Let $q \geq 2, n, t, k$ and $i$ be integers such that $1 \leq t < n$, $1\leq i \leq t$, and $1 \leq k \leq \min\{n-1, 2t\}$. Let 
	$$C_{k,i}:=q-\ceil*{\frac{k}{t}}+\ceil*{\frac{k_i}{t}}+\ceil*{\frac{iq-k}{n}},$$
	where $k_i=(k-iq)\bmod{n}$. 
	\begin{itemize}
	\item If $1 \leq k \leq t$, then
	\begin{equation}\label{eq2}
		\min_{1\leq i \leq t}\{q, C_{k, i}\}=\left\{\begin{array}{ll}
		q-1, & \text{if } q\mid k,\\
		q, & \text{otherwise.}
		\end{array}\right.
		\end{equation}
		\item If $t< k \leq \min\{n-1, 2t\}$, then
		\begin{equation}\label{eq3}
		\min_{1\leq i \leq t}\{q, C_{k, i}\}=\left\{\begin{array}{ll}
		q-2, & \text{if } q\mid k,\\	
		q-1, & \text{if }t\geq (k\bmod{q}) \text{ and }q\nmid k, \\
		q-1, & \text{if }q\mid (k+n), \, k\leq qt-n, \, \text{and }q\nmid k,\\
		q, & \text{otherwise.}
		\end{array}\right.
		\end{equation}
		\end{itemize}
	\end{proposition}
		\begin{proof}
		First, suppose that $1 \leq k \leq t$. Then, we have $C_{k,i}=q-1+\ceil*{\frac{k_i}{t}}+\ceil*{\frac{iq-k}{n}} \geq q-1$. Thus,
		\begin{align*}
		C_{k,i}=q-1 
		&\Leftrightarrow \ceil*{\frac{k_i}{t}}+\ceil*{\frac{iq-k}{n}}=0\\
		&\Leftrightarrow \ceil*{\frac{k_i}{t}}=0\text{ and }\ceil*{\frac{iq-k}{n}}=0\\
		&\Leftrightarrow k_i=0\text{ and }0\leq k-iq < n\\
		&\Leftrightarrow k=iq.
		\end{align*}
		This implies that, for each $1\leq  k\leq t$ such that $q$ divides $k$, we can choose $i=\frac{k}{q}\leq \frac{t}{q}\leq t$ to obtain $C_{k, i}=q-1$, and (\ref{eq2}) follows.
		
		Now, for each $t< k\leq \min\{n-1, 2t\}$, we have $C_{k,i}=q-2+\ceil*{\frac{k_i}{t}}+\ceil*{\frac{iq-k}{n}}\geq q-2$. Hence,
		$C_{k, i}=q-2$ if and only if $ k=iq.$ Thus, for each $t< k\leq \min\{n-1, 2t\}$ such that $q$ divides $k$, we can choose $i=\frac{k}{q}\leq \frac{2t}{q}\leq t$ to obtain $C_{k, i}=q-2$. On the other hand, if $q$ does not divide $k$ then
		\begin{align*}
			C_{k,i}=q-1 
			&\Leftrightarrow \ceil*{\frac{k_i}{t}}+\ceil*{\frac{iq-k}{n}}=1\\
			&\Leftrightarrow \left(\ceil*{\frac{k_i}{t}}=1 \text{ and } \ceil*{\frac{iq-k}{n}}=0\right)\text{ or }\left( \ceil*{\frac{k_i}{t}}=0 \text{ and } \ceil*{\frac{iq-k}{n}}=1\right).
		\end{align*}
		
		For the first case, we have
		\begin{align*}
			\ceil*{\frac{k_i}{t}}=1 \text{ and } \ceil*{\frac{iq-k}{n}}=0 
			&\Leftrightarrow 0<k_i\leq t \text{ and }0\leq k-iq<n\\
			&\Leftrightarrow 0< k-iq\leq t\\
			& 
			\Leftrightarrow \dfrac{k-t}{q} \leq i < \dfrac{k}{q}\\
			&\Leftrightarrow \ceil*{\dfrac{k-t}{q}} \leq i \leq \ceil*{\dfrac{k}{q}} - 1.
		\end{align*}
		Since $\ceil*{\frac{k}{q}}-\ceil*{\frac{k-t}{q}}\geq 1$ if and only if $t\geq (k\bmod{q})$, we have that, for each $t< k\leq \min\{n-1, 2t\}$ such that $q\nmid k$ and $t\geq (k\bmod{q})$, we can choose any $ \ceil*{\dfrac{k-t}{q}} \leq i \leq \ceil*{\dfrac{k}{q}} - 1$ to obtain $C_{k, i}=q-1$.
		
		For the second case,
		\begin{align*}
			\ceil*{\frac{k_i}{t}}=0 \text{ and } \ceil*{\frac{iq-k}{n}}=1 
			&\Leftrightarrow k_i=0 \text{ and }-n\leq k-iq< 0\\
			&\Leftrightarrow k+n=iq.
		\end{align*}
		Now, since $1 \leq i \leq t$, we must have $\dfrac{k+n}{q} \leq t$, that is, $k \leq qt - n$.
		Thus we have that, for each $t< k \leq \min\{n-1, 2t\} $ such that $q\nmid k$, $q\mid (k+n)$, and $k\leq qt-n$, we can choose $i=\frac{k+n}{q}$ to obtain $C_{k, i}=q-1$. 
		
		This implies that for each $t< k \leq \min\{n-1, 2t\}$, 
		\begin{equation*}
		\min_{1\leq i \leq t}\{q, C_{k, i}\}=\left\{\begin{array}{ll}
		q-2, & \text{if } q\mid k,\\	
		q-1, & \text{if }t\geq (k\bmod{q}) \text{ and }q\nmid k, \\
		q-1, & \text{if }q\mid (k+n), \, k\leq qt-n, \, \text{and }q\nmid k,\\
		q, & \text{otherwise.}
		\end{array}\right.
		\end{equation*}
		\end{proof}

Using the results in  Proposition \ref{valores Cik}, we obtain the following corollary.
		
		\begin{corollary}  Let $S=\langle n, n+1, \dots, n+t \rangle$ with $\ceil*{\textstyle\frac{n-1}{2}}\leq t < n$. Then 
			\begin{equation}\label{sum consecutive} 
				GM_q(S)=qn+\floor*{\frac{t}{q}}-2\ceil*{\frac{n}{q}}- \# A + 3, 
			\end{equation}
			where 
			$$
			A=\{t<k\leq n-1: t\geq ( k\bmod{q}), \,  q\nmid k\}\cup\{t<k\leq n-1: k\leq qt-n, \, q|(k+n), \, \text{and }q\nmid k\}.
			$$
		\end{corollary}
		\begin{proof}We have $GM_q(S)=\sum_{k=0}^{n-1}\min_{1\leq i \leq t}\{q, C_{k, i}\}$.
				From Equation (\ref{eq1}), we obtain $$\min_{1\leq i \leq t}\{q, C_{0, i}\} = q,$$ while Equation (\ref{eq2}) yields $$\sum_{k=1}^{t}\min_{1\leq i \leq t}\{q, C_{k, i}\} = \left(q-1\right)\floor*{\frac{t}{q}}+q\left(t-\floor*{\frac{t}{q}}\right).$$ So,
			\begin{equation}\label{eq4}
				\sum_{k=0}^{t}\min_{1\leq i \leq t}\{q, C_{k, i}\} = q(t+1)-\floor*{\frac{t}{q}}.
			\end{equation}
			Furthermore, we have  $n-1= \min\{n-1, 2t\}$,  and defining the set 
			\begin{align*}
			A&:= \{t<k\leq n-1 : t\geq (k\bmod{q})\text{ and } q\nmid k\}\\
			&\cup\{t<k\leq n-1 : k\leq qt-n, \, q|(k+n), \, \text{and }q\nmid k\},
			\end{align*}
			 by Equation (\ref{eq3}), we  conclude that 
			\begin{equation}\label{eq5}
			\begin{split}
			\sum_{k=t+1}^{n-1}\min_{1\leq i \leq t} \{q, C_{k, i}\} & = (q-2)\left( \floor*{\frac{n-1}{q}}-\floor*{\frac{t}{q}}\right)\\
				&+(q-1)\# A +q\left( n-1-t-\floor*{\frac{n-1}{q}}+\floor*{\frac{t}{q}}-\#A \right).
				\end{split}
			\end{equation}
			The result follows from Equations (\ref{eq4}) and (\ref{eq5}).
		\end{proof}
		Using the last previous results, we are able to obtain explicit expressions for the GM bound $GM_q(S)$ for $2\leq q$, $S=\langle n, n+1, \dots, n+t \rangle$  and $\ceil*{\textstyle\frac{n-1}{2}}\leq t \leq n-1$.  Firstly, we observe that, if $S$ is an ordinary semigruoup (i.e. $t=n-1$), then $\# A=0$. Therefore, by Equation (\ref{sum consecutive}),  $$GM_q(S)=2+qn-\ceil*{\frac{n}{q}}.$$ This formula matches the one provided in \cite[Example 3]{GM2009}. On the other hand, if $t=1$ then $S=\langle n, n+1\rangle$ and consequently the formula for $GM_q(S)$ reduces to that given in Corollary \ref{GMq(S)_twogenerators}. Henceforth, we may assume $t\neq 1$.

\begin{theorem} \label{GM formulas fechadas}
Let $q \geq 2$, and $S=\langle n, n+1, \dots, n+t \rangle$, where  $\ceil*{\textstyle\frac{n-1}{2}}\leq t < n$ and $t\neq 1$. 
\begin{enumerate}
\item If $q\leq t+1$, then 
		$$
		GM_q(S)=3+(q-1)n-\ceil*{\frac{n}{q}}+t.
		$$
\item If $t+2\leq q \leq n-1$, then 
		$$
		GM_q(S)=4+(q-1)n+q-2\ceil*{\frac{n}{q}}+\floor*{\frac{t+n}{q}}-\max\left\{\floor*{\frac{t+n}{q}}, \min\left\{1+\floor*{\frac{n}{q}}, t \right\}\right\}.
		$$		
\item If $n\leq q$, then 
$$
GM_q(S)=2+qn-\ceil*{\frac{2n}{q}}+\floor*{\frac{t+n}{q}}.
$$
\end{enumerate}
\end{theorem}
\begin{proof}
		$(1)$ If $q\leq t+1$ then, for any $k \in \{t+1, t+2, \ldots, n-1\}$ we have $t\geq (k\bmod{q})$, and therefore $A=\{t<k\leq n-1 : q\nmid k\}$. Thus
		$$
		\#A=n-t-\ceil*{\frac{n}{q}}+\floor*{\frac{t}{q}},
		$$
and the result follows from Equation (\ref{sum consecutive}).
		
		$(2)$ For $t+2\leq q \leq n-1$ we have  
		$$\{t<k\leq n-1: t\geq k\bmod{q}, \, q\nmid k\}=\{q+1, q+2, \dots, n-1\}.$$
		On the other hand, if $q|(k+n)$, then $q\nmid k$ if and only if $q\nmid n$. Moreover, as $\ceil*{\textstyle\frac{n-1}{2}}\leq t$ and $t+2\leq q \leq n-1$, we obtain $\frac{n}{2}<q\leq n-1$ and therefore $q\nmid n$. Thus,
		$$
		\{t<k\leq n-1: k\leq qt-n, \, q|(k+n), \, \text{and }q\nmid k\}=\{t<k\leq \min\{n-1, qt-n\}:  q|(k+n)\}.
		$$ 
		So,
		\begin{align*}
		A&=\{q+1, q+2, \dots, n-1\}\cup \{t< k\leq \min\{n-1, qt-n\}: q| (k+n)\}\\
		&=\{q+1, q+2, \dots, n-1\}\cupdot \{t<  k\leq \min\{q, qt-n\}: q| (k+n)\}
		\end{align*}
		and therefore
		\begin{align*}
		\# A&=n-1-q+\max\left\{0, \min\left\{1+\floor*{\frac{n}{q}}, t \right\}-\floor*{\frac{t+n}{q}}\right\}\\
		&=n-1-q-\floor*{\frac{t+n}{q}}+\max\left\{\floor*{\frac{t+n}{q}}, \min\left\{1+\floor*{\frac{n}{q}}, t \right\}\right\}.
		\end{align*}

		$(3)$ Suppose that $n\leq q$. Thus, 
		$$\{t<k\leq n-1: t\geq (k\bmod{q}), \, q\nmid k\}=\{t<k\leq n-1: t\geq k, \, q\nmid k\}=\emptyset.$$
		In addition, since $t\ne 1$, we have $qt-n\geq nt-n\geq n > k$ and therefore
		$$
		\{t<k\leq n-1: k\leq qt-n, \, q\mid k+n, \, \text{and }q\nmid k\}=\{t<k\leq n-1: q| (k+n)\}.
		$$
		Thus, $A=\{t<k\leq n-1: q|(k+n)\}$ and $$\# A=\floor*{\frac{2n-1}{q}}-\floor*{\frac{t+n}{q}}.$$ 
		The result follows from Equation (\ref{sum consecutive}).
\end{proof}

From the proof of Proposition \ref{valores Cik}, we observe that finding a simple closed formula for the Geil-Matsumoto bound becomes more challenging for the semigroup $S=\langle n, n+1, \dots, n+t\rangle$ when $t<\textstyle\ceil*{\frac{n-1}{2}}$, as the number of cases grows substantially.
Notably, for smaller values of $t$, the Geil-Matsumoto bound yields more significant improvements over Lewittes' bound. This is exemplified in Figure \ref{figure1}, which compares the two bounds for $n=15$, $1\leq q \leq 75$, and $t=3, 6, 9, 12$. For $t=9, 12$, we use the formula from Theorem \ref{GM formulas fechadas}, whereas for $t=3, 6$, we apply Theorem \ref{teo_GMbound consecutivos}. As the figure shows, the most pronounced improvements arise when $t$ is small.		
		
\begin{figure}[htbp]
	\centering
	\begin{tikzpicture}
		\begin{axis}[
			width=1.0\textwidth,
			height=10cm,
			xlabel={$q$},
			ylabel={$L_q(S) - GM_q(S)$},
			legend style={at={(0.97,0.94)}, anchor=north east},
			legend cell align={left},
			xtick={0,10,...,80},
			ytick={0,10,...,30},
			mark size=2pt,
			xmin=-2,    
			ymin=-2,     
			enlarge x limits={upper, value=0.03},  
			enlarge y limits={upper, value=0.05}
			]
			\addplot[
			color=black,
			mark=o
			] coordinates {
				(1,14) (2,22) (3,27) (4,29) (5,25) (6,24) (7,23) (8,19) (9,18) (10,20) (11,21) (12,18) (13,16) (14,11) (15,0) (16,0) (17,0) (18,0) (19,8) (20,10) (21,12) (22,13) (23,12) (24,12) (25,12) (26,11) (27,9) (28,8) (29,6) (30,0) (31,0) (32,0) (33,0) (34,0) (35,0) (36,0) (37,4) (38,5) (39,6) (40,7) (41,7) (42,6) (43,5) (44,3) (45,0) (46,0) (47,0) (48,0) (49,0) (50,0) (51,0) (52,0) (53,0) (54,0) (55,2) (56,3) (57,3) (58,3) (59,2) (60,0) (61,0) (62,0) (63,0) (64,0) (65,0) (66,0) (67,0) (68,0) (69,0) (70,0) (71,0) (72,0) (73,1) (74,1) (75,0)
			};
			\addplot[
			color=red,
			mark=o
			] coordinates {
				(1,14) (2,16) (3,14) (4,13) (5,12) (6,12) (7,12) (8,10) (9,9) (10,8) (11,8) (12,7) (13,6) (14,5) (15,0) (16,0) (17,0) (18,0) (19,0) (20,0) (21,0) (22,3) (23,3) (24,3) (25,3) (26,3) (27,3) (28,3) (29,2) (30,0) (31,0) (32,0) (33,0) (34,0) (35,0) (36,0) (37,0) (38,0) (39,0) (40,0) (41,0) (42,0) (43,1) (44,1) (45,0) (46,0) (47,0) (48,0) (49,0) (50,0) (51,0) (52,0) (53,0) (54,0) (55,0) (56,0) (57,0) (58,0) (59,0) (60,0) (61,0) (62,0) (63,0) (64,0) (65,0) (66,0) (67,0) (68,0) (69,0) (70,0) (71,0) (72,0)
				(73,0) (74,0) (75,0)
			};
			\addplot[
			color=blue,
			mark=o
			] coordinates {
				(1,14) (2,12) (3,9) (4,8) (5,7) (6,7) (7,7) (8,6) (9,6) (10,6) (11,5) (12,4) (13,4) (14,3) (15,0) (16,0) (17,0) (18,0) (19,0) (20,0) (21,0) (22,0) (23,0) (24,0) (25,1) (26,1) (27,1) (28,1) (29,1) (30,0) (31,0) (32,0) (33,0) (34,0) (35,0) (36,0) (37,0) (38,0) (39,0) (40,0) (41,0) (42,0) (43,0) (44,0) (45,0) (46,0) (47,0) (48,0) (49,0) (50,0) (51,0) (52,0) (53,0) (54,0) (55,0) (56,0) (57,0) (58,0) (59,0) (60,0) (61,0) (62,0) (63,0) (64,0) (65,0) (66,0) (67,0) (68,0) (69,0) (70,0) (71,0) (72,0)
				(73,0) (74,0) (75,0)
			};
			\addplot[
			color=green!60!black,
			mark=o
			] coordinates {
				(1,14) (2,9) (3,6) (4,5) (5,4) (6,4) (7,4) (8,3) (9,3) (10,3) (11,3) (12,3) (13,3) (14,3) (15,0) (16,0) (17,0) (18,0) (19,0) (20,0) (21,0) (22,0) (23,0) (24,0) (25,0) (26,0) (27,0) (28,1) (29,1) (30,0) (31,0) (32,0) (33,0) (34,0) (35,0) (36,0) (37,0) (38,0) (39,0) (40,0) (41,0) (42,0) (43,0) (44,0) (45,0) (46,0) (47,0) (48,0) (49,0) (50,0) (51,0) (52,0) (53,0) (54,0) (55,0) (56,0) (57,0) (58,0) (59,0) (60,0) (61,0) (62,0) (63,0) (64,0) (65,0) (66,0) (67,0) (68,0) (69,0) (70,0) (71,0) (72,0)
				(73,0) (74,0) (75,0)
			};
			\legend{$t=3$, $t=6$, $t=9$, $t=12$}
		\end{axis}
	\end{tikzpicture}
	\caption{Analysis of the improvement of $GM_q(S)$ with respect to $L_q(S)$ for the semigroup $S = \langle n, n+1, \dots, n+t \rangle$, with $n = 15$ and $1 \leq q \leq 75$.}\label{figure1}
\end{figure}
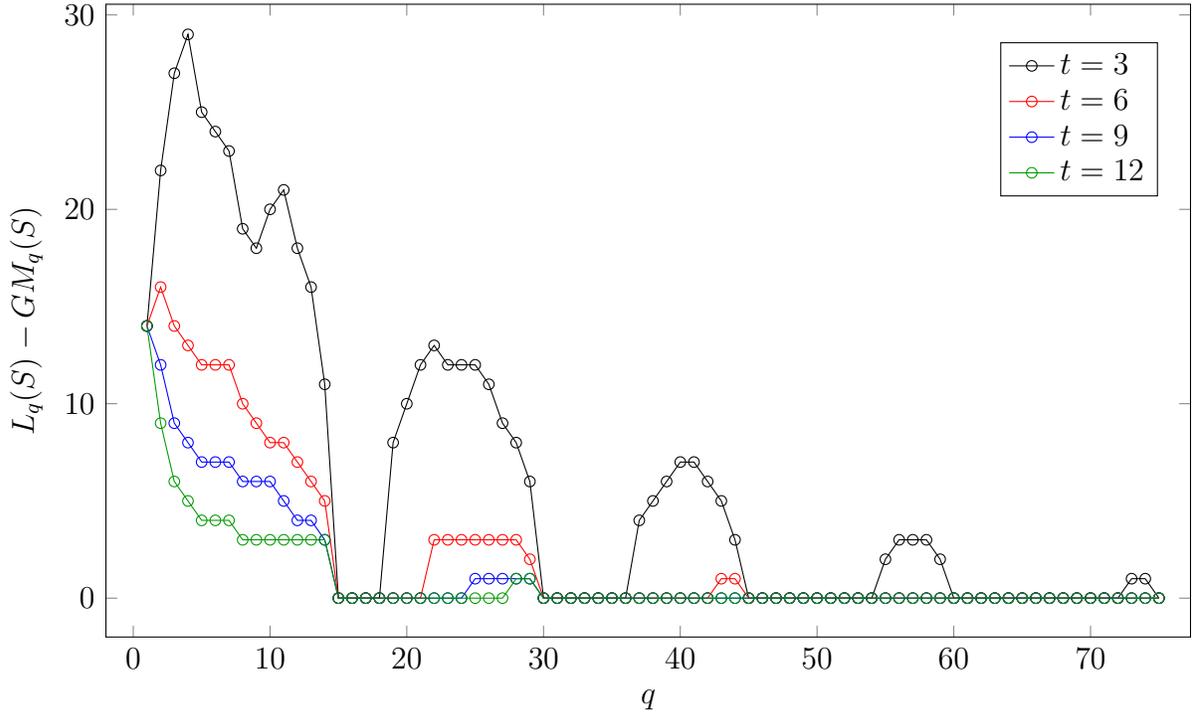

	\section{Geil-Matsumoto bound and rational points on Kummer extensions}
			Let $\mathcal{X}$ be the algebraic curve defined over $\fq$ and given by the affine equation 			
			\begin{equation} \label{Kummer}
			\mathcal{X} : y^m = f(x) = \prod_{k=1}^{r} (x - \alpha_k)^{\la_k},\quad 1\leq \la_k \leq m-1, \quad \la_k\in \N,
			\end{equation}
			where $m, r\geq 2$, $\char(\fq) \nmid m$, $f(x) \in \fq[x]$ is not a $d$-th power (where $d$ divides $m$) of an element in $\fq(x)$, and $\alpha_1, \ldots , \alpha_r \in \fq$ are pairwise distinct elements. Let $\la_0:=-\textstyle\sum_{k=1}^{r}\la_k$ and  $\fq(\mathcal{X}):=\fq(x,y)$ be the function field of the curve. The field extension $\fq(\cX)/\fq(x)$ is called a Kummer extension.
			For each $j=1,\ldots , r$, let $P_j$ and $P_0$ be the places in $\fq(x)$ corresponding to the zero of $x - \alpha_j$ and the pole of $x$, respectively. 	
		For $0\leq s \leq r$ such that $(m, \la_s)=1$, let $Q_s$ be the unique place in $\fq(\mathcal{X})$ lying over $P_s$.
		
Regarding rational points on Kummer extensions, Yoo and Lee established an upper bound (the so called YL bound) for algebraic curves of Kummer type. In some of the cases they actually improved Hasse-Weil bound, see \cite{YL2025}.

\begin{corollary}\cite[Corollary 4.1]{YL2025} \label{YLbound}
	Consider a Kummer extension with affine defining equation given by $\cX: y^m =f(x)$ as in (\ref{Kummer}).
	Then the following bound holds:
	\[
	|N_q(\cX) - (q + 1)| \leq 
	\begin{cases}
		(q - r)(m - 1), & \text{if } (m, \lambda_0) = 1, \\
		(q + 1 - r)(m - 1), & \text{if } (m, \lambda_0) \neq 1.
	\end{cases}
	\]
\end{corollary}

First we provide some examples of algebraic curves over $\fq$  in which the GM bound  given in Theorem \ref{teo_GMbound} improves the Lewittes, HWS, YL, and  Ihara bounds. To establish this, we require a description of the Apéry set of the Weierstrass semigroup  in a point of the curve. For Kummer extensions as defined in (\ref{Kummer}), Cotterill, Mendoza, and Speziali \cite{CMS2025} determined the Apéry set of $m$ in $H(Q_s)$ and a system of generators for $H(Q_s)$ for any $0\leq s \leq r$ such that $(m, \la_s)=1$. More precisely, they proved the following result.
\begin{corollary}\label{coro_generators}\cite[Corollary 3.9]{CMS2025}
	Let $0\leq s\leq r$ be such that $(m, \la_s)=1$. The Apéry set of $m$ in the Weierstrass semigroup $H(Q_s)$ is given by
	$$
	\Ap(H(Q_s), m)=\{0\}\cup\{m\beta(i)+t_s(i): 1\leq i \leq m-1\},
	$$
	where $\beta(i)=\textstyle\sum_{k=0}^{r}\ceil*{\frac{i\la_k}{m}}-1$ and $t_s(i)=(i\la_s)\bmod{m}$ for $1\leq i\leq m-1$.
	In particular,
	$$
	H(Q_s)=\langle m, m\beta(i)+t_s(i): 1 \leq i \leq m-1\rangle.
	$$
\end{corollary}
We use this last corollary to identify some cases where the Geil-Matsumoto and  Lewittes bounds coincide for Kummer extensions as in (\ref{Kummer}). From Corollary \ref{coro_generators}, if $\beta(i)\geq 1$ for every $1\leq i\leq m-1$ and $m$ divides $q-1$,  we have  $m_{H(Q)}=m$ and, from Proposition \ref{igualdade cotas} we conclude $GM_q(H(Q)) = L_q(H(Q))$ for any totally ramified place $Q$ in the extension $\fq(\cX)/\fq(x)$.

Now, from Theorem~\ref{teo_GMbound}  and Corollary~\ref{coro_generators}, we construct families of algebraic curves where the Geil--Matsumoto bound outperforms known upper bounds for the number of rational places. Consider the curve over $\mathbb{F}_q$ defined by  
\[
y^m = (x - \alpha_1)(x - \alpha_2)^{\lambda_2}(x - \alpha_3)^{\lambda_3},
\]  
where $(m, q) = 1$ and $\alpha_1, \alpha_2, \alpha_3 \in \mathbb{F}_q$ are distinct, and let $Q_1$ be the rational place of $\mathbb{F}_q(x, y)$ corresponding to $x = \alpha_1$. In Table~\ref{table2} we compare the GM bound for $H(Q_1)$ against the Lewittes, HWS, YL, and Ihara bounds for various $q, m, \lambda_2, \lambda_3$, demonstrating its improvement in these cases. 

In an analogous way, considering a curve over $\mathbb{F}_q$ defined by  
\[
y^m = (x - \alpha_1)(x - \alpha_2)^{\lambda_2}(x - \alpha_3)^{\lambda_3}(x - \alpha_4)^{\lambda_4},
\]  
in Table~\ref{table3} we compare the GM bound for $H(Q_1)$ against the Lewittes, HWS, YL, and Ihara bounds for various $q, m, \lambda_2, \lambda_3, \lambda_4$, demonstrating its improvement in these cases.

	\begin{table}[h!]\label{goodbounds}
		\centering
		\begin{tabular}{|c|c|c|c|c|c|c|c|c|c|c|}\hline
			$q$ & $g$ & $m$ & $\la_2$ & $\la_3$ & $S:=H(Q_1)$& $GM_q(S)$ & $L_q(S)$ & $\text{HWS }$ & $\text{YL }$ & $\text{Ihara}$  \\ \hline
			$ 37 $ & $ 24 $ & $ 27 $ & $ 3 $ & $ 3 $ & $\langle 8, 9, 31\rangle$ & $295 $ & $ 297 $ & $ 326 $ & $ 922 $ & $ 299 $ \\ \hline
			$ 37 $ & $ 28 $ & $ 32 $ & $ 7 $ & $ 28 $ & $\langle 8, 9\rangle$ & $294 $ & $ 297 $ & $ 374 $ & $ 1123 $ & $ 333 $ \\ \hline
			$ 37 $ & $ 33 $ & $ 34 $ & $ 3 $ & $ 3 $ & $\langle 10, 11, 34, 39\rangle$ & $366 $ & $ 371 $ & $ 434 $ & $ 1160 $ & $ 374 $ \\ \hline
			$ 41 $ & $ 21 $ & $ 24 $ & $ 3 $ & $ 3 $ & $\langle 7, 8\rangle$ & $287 $ & $ 288 $ & $ 294 $ & $ 916 $ & $ 297 $ \\ \hline
			$ 41 $ & $ 31 $ & $ 35 $ & $ 7 $ & $ 31 $ & $\langle 9, 10, 35\rangle$ & $367 $ & $ 370 $ & $ 414 $ & $ 1334 $ & $ 386 $ \\ \hline
			$ 41 $ & $ 36 $ & $ 40 $ & $ 4 $ & $ 4 $ & $\langle 9, 10\rangle$ & $366 $ & $ 370 $ & $ 474 $ & $ 1524 $ & $ 430 $ \\ \hline
			$ 79 $ & $ 45 $ & $ 50 $ & $ 9 $ & $ 45 $ & $\langle 10, 11\rangle$ & $789 $ & $ 791 $ & $ 845 $ & $ 3853 $ & $ 830 $ \\ \hline
		\end{tabular}
		\captionof{table}{Improvements with respect to the Lewittes bound ({\bf L}), the HWS bound ({\bf HWS}), the YL bound ({\bf YL}), and the Ihara bound ({\bf Ihara}).  }\label{table2}
	\end{table}

\lu{	\begin{table}[h!]\label{goodbounds}
		\centering
		\begin{tabular}{|c|c|c|c|c|c|c|c|c|c|c|c|}\hline
			$q$ & $g$ & $m$ & $\la_2$ & $\la_3$ & $\la_4$ & $S:=H(Q_1)$ & $GM_q(S)$ & $L_q(S)$ & $\text{HWS }$ & $\text{YL }$ & $\text{Ihara}$  \\ \hline
			$ 13 $ & $ 10 $ & $ 10 $ & $ 5 $ & $ 8 $ & $8$ & $\langle 5, 6\rangle$ & $64 $ & $ 66 $ & $ 84 $ & $ 104 $ & $ 73 $ \\ \hline
			$ 13 $ & $ 15 $ & $ 11 $ & $ 3 $ & $ 6 $ & $6$ & $\langle 7,9,11\rangle$ & $91 $ & $ 92 $ & $ 119 $ & $ 104 $ & $ 97 $ \\ \hline
		\end{tabular}
		\captionof{table}{Improvements with respect to the Lewittes bound ({\bf L}), the HWS bound ({\bf HWS}), the YL bound ({\bf YL}), and the Ihara bound ({\bf Ihara}).  }\label{table3}
	\end{table}}
	  
We now  provide some upper bounds on the number of rational points for curves admitting a rational point with Weierstrass semigroup generated by consecutive integers. In the next result from \cite{CMQ2016}, the authors provided a family of algebraic curves satisfying this property. This family includes the famous Hermitian curve, and more generally curves of type norm-trace, see \cite{G2003}. The upper bound for the number of rational points on this family of curves will follow straightforward from  Theorem \ref{GM formulas fechadas}.
	
\begin{theorem}\cite[Theorem 3.4]{CMQ2016}\label{CMQ}
Consider a Kummer extension defined by the algebraic curve $\cX: \, y^m=f(x)$,  where $f(x)$ is a separable polynomial of degree $r\geq 2$ and $m=rt +1$ with $t \geq 1$. If $Q\neq Q_0 $ is a totally ramified place of $\fq(x, y)/\fq(x)$, then
\begin{equation*}\label{hp}
H(Q)=\langle m-\lfloor m/r\rfloor, m-\lfloor m/r\rfloor+1, \dots, m\rangle.
\end{equation*}
\end{theorem}

We establish an upper bound for the number of rational points over $\fq$ on the curves from Theorem~\ref{CMQ} by applying Theorem~\ref{GM formulas fechadas}. The condition on the number of generators in Theorem~\ref{GM formulas fechadas} implies that
$$ \ceil*{\frac{m-1-\floor*{m/r}}{2}} \leq \floor*{\frac{m}{r}}  \leq m-1-\floor*{\frac{m}{r}},$$
which is equivalent to
$$
2\floor*{\frac{m}{r}}\leq m-1 \leq 3\floor*{\frac{m}{r}}.
$$
This last inequality holds only for $r=2$ or $3$. Thus, we can provided the following upper bound for the number of rational points.

\begin{proposition}\label{prop_kummer_consecutive_bound}
Let $\fq(\cX)/\fq$ be a  Kummer extension defined by the algebraic curve $\cX:\, y^m=f(x)$, where $2 \leq m=rt+1 \leq q-1$ and $f(x)$ is a separable polynomial of degree $r=2$ or $3$. Then the  number $N_q(\cX)$ of $\fq$-rational points of the curve $\cX$ satisfies
$$ N_q(\cX) \leq 2+\frac{q(mr-m+1)}{r}-\ceil*{\frac{2(mr-m+1)}{qr}}.$$
\end{proposition}
\begin{proof} We apply Theorem \ref{CMQ}, item $(3)$ of Theorem \ref{GM formulas fechadas}, and the fact that $\floor*{\textstyle\frac{m}{r}} = \textstyle\frac{m-1}{r}$.
\end{proof}

	

		
		
	\end{document}